\documentclass[intlimits]{aptpub}
\usepackage{cite}
\usepackage[short labels]{enumitem}

\numberwithin{equation}{section}

\newenvironment{sublemma}{\begin{enumerate}[i), ref=\thelemma(\roman*)]}{\end{enumerate}}
\newenvironment{subtheorem}{\begin{enumerate}[i), ref=\thetheorem(\roman*)]}{\end{enumerate}}
\newenvironment{subcorollary}{\begin{enumerate}[i), ref=\thecorollary(\roman*)]}{\end{enumerate}}

\renewcommand{\vec}[1]{\mathbf{#1}}
\newcommand{\rising}[1]{\overline{#1}}

\DeclareMathOperator{\Ex}{\mathbb{E}}
\DeclareMathOperator{\Var}{Var}

\renewcommand{\epsilon}{\varepsilon}

\newcommand{\stirling}[2]{\genfrac{\{}{\}}{0pt}{}{#1}{#2}}
\newcommand{\st}{:}
\newcommand{\transpose}{{\mathrm{T}}}
\DeclareMathOperator{\Supp}{Supp}
\newcommand{\koutmaps}[2]{\ensuremath{\mathcal{M}_{#1,#2}}}
\newcommand{\koutmapsdegs}[3]{\ensuremath{\mathcal{M}_{#1,#2}(#3)}}

\newcommand{\digraph}[3]{\ensuremath{M_{#1,#2}^{#3}}}

\newcommand{\indegseq}[2]{\ensuremath{\vec{D}_{#1}^{#2}}}
\newcommand{\indeg}[3]{\ensuremath{D_{#1,#2}^{#3}}}

\DeclareMathOperator{\dist}{d}
\newcommand{\totalvar}{\ensuremath{\dist_{TV}(\digraph{n}{k}{\alpha},\digraph{n}{k}{\infty})}}
\newcommand{\totalvarsquares}{\ensuremath{\dist_{TV}\Bigl(\sum_{j=1}^{n}(\indeg{n}{j}{\alpha})^2,\sum_{j=1}^{n}(\indeg{n}{j}{\infty})^2\Bigr)}}

\authornames{N. R. Peterson and B. Pittel}
\shorttitle{Distance Between Two Random $k$-Out Digraphs}

\begin{document}

\title{Distance Between Two Random \lowercase{$k$}-Out Digraphs, \linebreak with and without Preferential Attachment}

\authorone[The Ohio State University]{Nicholas R. Peterson}
\authorone[The Ohio State University]{Boris Pittel}
\addressone{100 Math Tower; 231 W 18th Ave; Columbus, OH 43210; USA\\
The authors gratefully acknowledge support from NSF grant \# DMS-1101237.}

\begin{abstract}
A random $k$-out mapping (digraph) on $[n]$ is generated by choosing $k$ random images of each vertex one at a time, subject to a ``preferential attachment'' rule: the current vertex selects an image $i$ with probability proportional to a given parameter $\alpha=\alpha(n)$ plus the number of times $i$ has already been selected.  Intuitively, the larger $\alpha$ gets, the closer the resulting $k$-out mapping is to the uniformly random $k$-out mapping.  We prove that $\alpha=\Theta(n^{1/2})$ is the threshold for $\alpha$ growing ``fast enough'' to make the random digraph approach the uniformly random digraph in terms of the total variation distance.  We also determine an exact limit for this distance for $\alpha=\beta n^{1/2}$.
\end{abstract}



\section{Introduction}\label{sec:intro}
In the study of random graph/digraph processes, {\em preferential attachment} (or the popularity effect) refers broadly to processes in which edges or arcs are inserted one at a time, and vertices chosen as endpoints previously are more likely to be chosen going forward. These processes have become well known since Barab\'asi and Albert \cite{barabasi-albert-emergenceofscaling} introduced the first such model to explain a ``scale-free'' vertex degree distribution observed empirically in various real-world networks. In this scheme the vertex set grows in time, each new vertex attaching itself randomly to existing vertices, with probabilities proportional to their current ``popularity'' (degree). The Barab\'asi-Albert model was later formalized, and studied rigorously, in papers by Bollob\'as and Riordan \cite{bollobas-riordan-scalefree} and Bollob\'as, Riordan, Spencer, and Tusn\'ady \cite{bollobas-riordan-spencer-tusnady-scalefree}.

For a single host vertex, the resulting graph (a non-uniform recursive tree) had been studied some years earlier; see Bergeron et al. \cite{bergeron-flajolet-salvy-increasingtrees}, Mahmoud et al. \cite{mahmoud-smythe-szymanski-recursivetrees}, and Pittel \cite{pittel-recursivetrees}, for instance. 

Since then, a wealth of preferential attachment graph models have been studied, see, for example, Buckley and Osthus \cite{buckley-osthus-popularity},   
 Bollob\'as et al. \cite{bollobas-borgs-chayes-riordan-directedscalefreegraphs}, and
 Deijfen \cite{deijfen-preferential}.
 

Recently Pittel \cite{pittel-evolvingbydegrees} studied a graph process $\{G_{\alpha}(n,M)\}_{M=0}^{N}$, which is a ``preferential attachment'' counterpart of the Erd\"os-R\'enyi process $\{G(n,M)\}_{M=0}^{N}$ on a {\it fixed\/} vertex set $[n]$, ($N:=\binom{n}{2}$): given the current graph $G_{\alpha}(n,M)$, $G_{\alpha}(n,M+1)$ is obtained by adding a new edge; we choose to connect currently non-adjacent vertices $i$ and $j$ with probability proportional to $(d_i+\alpha)(d_j+\alpha)$, where $d_i,d_j$ are the degrees of vertices $i$ and $j$ in $G_{\alpha}(n,M)$. Clearly, $\{G(n,M)\}_{M=0}^{n}$ is the limiting case $\alpha=\infty$ of $\{G_\alpha(n<M)\}_{M=0}^{N}$.  The main result in \cite{pittel-evolvingbydegrees} is that w.h.p. $G_{\alpha}(n,M)$ develops a giant component when the average vertex degree $c:=2M/n$ exceeds $\frac{\alpha}{\alpha+1}$, and the giant component has size asymptotic to $n\left[1-\left(\frac{\alpha+c^*}{\alpha+c}\right)^{\alpha}\right]$, where $c^*<\frac{\alpha}{\alpha+1}$ is a root of
\begin{equation*}
\frac{c}{(\alpha+c)^{2+\alpha}}=\frac{c^*}{(\alpha+c^*)^{2+\alpha}}.
\end{equation*}
Notably, formally letting $\alpha=\infty$ in this result recovers the result of Erd\"os and R\`enyi \cite{erdos-renyi-evolution}, that the Erd\"os-R\`enyi process $\{G(n,M)\}$ develops a giant component when $c:=2M/n$ exceeds 1, and that the giant component has size $n(1-\frac{c^*}{c})$, where $c^*<1$ satisfies $c^*e^{-c^*}=ce^{-c}$.

Another model of this type is a preferential attachment model for random mappings, defined and studied by Hansen and Jaworski \cite{hansen-jaworski-exchindegrees,hansen-jaworski-prefattach,hansen-jaworski-local}.
Let $\alpha>0$, and say that each vertex in $[n]$ has initial weight $\alpha$. The vertices take turns choosing their images, starting with $1$; conditioned on the previous steps in the process, vertex $i$ chooses vertex $j$ as its image with probability proportional to the current weight of vertex $j$, which then increases by 1. Call the resulting mapping $\digraph{n}{1}{\alpha}$ (this is our notation, not that of Hansen and Jaworski). 
The constant $\alpha$ measures, essentially, the ``independent-mindedness'' of the vertices as they choose their images: the larger $\alpha$ is, the less impact previous choices have on future ones. Letting $\alpha\to\infty$, we recover the uniformly random mapping $[n]\to[n]$. 
Extending earlier results of
Gertsbakh \cite{gertsbakh-epidemicprocess}, Burtin \cite{burtin-randommappings}, and Pittel \cite{pittel-randommappings} for the uniform mapping,  
Hansen and Jaworski \cite{hansen-jaworski-local} found the distributions of the
sets of ultimate ``successors'' and ``predecessors'' of a given set in the
random mapping $M_{n,1}^{\alpha}$.
Given this heuristic connection in the $\alpha=\infty$ case, it is natural to wonder: if we let $\alpha$ vary with $n$, and $\alpha\to\infty$ ``fast enough'' as $n\to\infty$, will $\digraph{n}{1}{\alpha}$ behave asymptotically like the uniform mapping?  If so, how fast is fast enough? In \cite{hansen-jaworski-prefattach}, Hansen and Jaworski established asymptotic properties of 
$\digraph{n}{1}{\alpha}$ in the case where $\alpha n\to\infty$, and specializing their results for the  
parameters studied in \cite{gertsbakh-epidemicprocess}, \cite{burtin-randommappings} and \cite{pittel-randommappings} in the case 
 $\alpha\to\infty$ does reveal the ``continuity at $\alpha=\infty$''.
At first glance, this might seem to indicate that $\alpha\to\infty$, however slowly, is enough to make $\digraph{n}{1}{\alpha}$ asymptotically uniform. However, we shall see in 
Section \ref{sec:smallalpha} that this is not the case. A rather simple parameter, the sum of squared
in-degrees, is much more sensitive to the behavior of $\alpha$, and its asymptotic distribution
is close to that for $\alpha=\infty$ only if $\alpha>> n^{1/2}$.

In this paper, we generalize Hansen and Jaworski's preferential attachment random mapping model to a new setting: $k$-valued mappings. Specifically, we study the collection of digraphs on vertex set $[n]$ in which each vertex has out-degree $k$, and the out-arcs belonging to each vertex are labeled $1,2,\ldots,k$. (Equivalently, these can be thought of as functional digraphs for mappings $[n]\to[n]^k$, where $[n]^k$ denotes the set of $k$-long vectors with coordinates in $[n]$.) We call our model $\digraph{n}{k}{\alpha}$, and the corresponding uniform model $\digraph{n}{k}{\infty}$; the case $k=1$ corresponds exactly to the model of Hansen and Jaworski.

Measuring the distance between $\digraph{n}{k}{\alpha}$ and $\digraph{n}{k}{\infty}$ via the total variation distance, we prove that $\alpha=\Theta(\sqrt{n})$ (notably, much smaller than $n$) is the threshold for ``fast enough'' growth to ensure asymptotic uniformity of $\digraph{n}{k}{\alpha}$.  We determine an exact limit for the distance in the case $\alpha=\beta\sqrt{n}$, where $\beta>0$ is fixed, and show that it is asymptotic to the distance between the distributions of the sum of squared in-degrees for $\digraph{n}{k}{\alpha}$ and $\digraph{n}{k}{\infty}$.


\section{Definition of the Model and Statement of Results}\label{sec:model}
Let $n,k\in\mathbb{N}$ and $\alpha\in(0,\infty)$. Let $\digraph{n}{k}{\alpha}$ be the directed multigraph on vertex set $[n]$ generated via insertion of a $kn$-long sequence of out-arcs, $k$ arcs per vertex, starting with the empty digraph and each vertex having {\em initial weight} $\alpha$.  At a generic step, choose (uniformly at random) a vertex with out-degree below $k$, and select its target vertex (image) with probability proportional to the target's current weight; increase the weight of the chosen vertex by 1.  After $kn$ steps, we arrive at a directed multigraph on vertex set $[n]$, in which each vertex has out-degree $k$ and its $k$ out-arcs are labeled {\em chronologically} $1,2,\ldots,k$.

While this scheme is perhaps a natural digraph growth process, the distribution of the terminal digraph is the same for any random ordering of the decision makers, provided that it depends only on the current out-degrees.  So, alternatively, we can consider the process consisting of $k$ rounds of the Hansen-Jaworski process \cite{hansen-jaworski-prefattach}, in which each round begins with the vertex weights accumulated during the previous rounds.  Think of this as a committee of $n$ people undergoing $k$ rounds of voting for a chair, in which votes are made publicly and people are swayed by the total votes in earlier rounds.
Given any $M:[n]\to[n]^k$, we find that
\begin{equation}\label{eq:distribution}
P(\digraph{n}{k}{\alpha}=M)=\frac{\prod_{j=1}^{n}\alpha^{\rising{d_j}}}{(\alpha n)^{\rising{kn}}},\qquad x^{\rising{y}}:=x(x+1)\cdots(x+y-1),
\end{equation}
where $(d_1,\ldots,d_n)$ is the in-degree sequence of $M$, including multiplicity.  

Analogously to Hansen and Jaworski's model, we can view the uniformly random mapping $[n]\to[n]^k$ as the limiting case of $\digraph{n}{k}{\alpha}$ in which $\alpha=\infty$: keeping $n$ fixed and allowing $\alpha$ to grow without bound, the current in-degree of each vertex becomes negligible compared to $\alpha$, so that all of the weights are nearly identical. In light of this connection, we let $\digraph{n}{k}{\infty}$ denote the uniformly random mapping $[n]\to[n]^k$.

Our main result is as follows:
\begin{theorem}\label{thm:totalvar}
Let $\totalvar$ denote the total variation distance between the measures on $\koutmaps{n}{k}$ (the collection of $k$-out maps on vertex set $[n]$) induced by $\digraph{n}{k}{\alpha}$ and $\digraph{n}{k}{\infty}$:
\begin{align*}
\totalvar&:=\sup_{\mathcal{A}\subseteq\koutmaps{n}{k}}\lvert P(\digraph{n}{k}{\alpha}\in\mathcal{A})-P(\digraph{n}{k}{\infty}\in\mathcal{A})\rvert\\
&=\frac{1}{2}\sum_{M\in\koutmaps{n}{k}}\lvert P(\digraph{n}{k}{\alpha}=M)-P(\digraph{n}{k}{\infty}=M)\rvert.
\end{align*}
Let $\alpha=\alpha(n)$ and $n\to\infty$.
\begin{subtheorem}
\item\label{thm:totalvar-bigalpha} If $\alpha/\sqrt{n}\to\infty$, then $\totalvar\to0$.
\item\label{thm:totalvar-smallalpha} If $\alpha\to\infty$ and $\alpha/\sqrt{n}\to0$ as $n\to\infty$, then $\totalvar\to1$.
\item\label{thm:totalvar-midalpha} If $\alpha=\beta\sqrt{n}$, $\beta\in(0,\infty)$ being fixed, then $\totalvar\to\frac{1}{2}\Ex\lvert1-\exp[-\mathcal{N}]\rvert$; here $\mathcal{N}$ is a Gaussian random variable with $\Ex[\mathcal{N}]=\frac{k^2}{4\beta^2}$ and $\Var[\mathcal{N}]=\frac{k^2}{2\beta^2}$.
\end{subtheorem}
\end{theorem}

Note that Theorem \ref{thm:totalvar} gives us a very strong result in the case where $\alpha\gg\sqrt{n}$: namely, that the difference in the probability assigned to any event $\mathcal A$ by the distributions of $\digraph{n}{k}{\alpha}$ and $\digraph{n}{k}{\infty}$ tends to $0$ with $n$. The result for $\alpha\ll\sqrt{n}$, on the other hand, is much less powerful: it simply tells us that {\em there is an event} $\mathcal A_n$ such that $P(\digraph{n}{1}{\alpha}\in \mathcal A_n)\to1$, while $P(\digraph{n}{1}{\infty}\in \mathcal A_n)\to0$.  As the example of the number of connected components in $\digraph{n}{1}{\alpha}$ discussed in Section 1 shows, there exist natural events whose probabilities under $\digraph{n}{k}{\alpha}$ and $\digraph{n}{k}{\infty}$ are nearly the same. Still,  Theorem \ref{thm:totalvar-smallalpha} is rather revealing:  it tells us that $\alpha=\Theta(\sqrt{n})$ is truly the threshold for {\em every} parameter of the $k$-out mapping having the same distribution in the limit for the random mappings $\digraph{n}{k}{\alpha}$ and $\digraph{n}{k}{\infty}$.

Theorem \ref{thm:totalvar} calls for finding a (hopefully natural) parameter $X$ of the $k$-out mapping such that the total variation distance $d_{TV}(X(\digraph{n}{k}{\alpha}), X(\digraph{n}{k}{\infty}))$ is asymptotic to $\totalvar$. This $X$ is a parameter whose distribution is most sensitive to finiteness of $\alpha$, allowed to be infinite only in the limit. We found such a parameter for the critical $\alpha=\Theta(n^{1/2})$.

\begin{theorem}\label{thm:sumofsquares} For a $k$-out mapping $M$, let $\mathbf{D}(M)=(D_1(M),\dots,D_n(M))$ denote the sequence of its in-degrees, and let $X(M):=\sum_i (D_i(M))^2$, the sum of squared in-degrees. Let $\alpha=\beta\sqrt{n}$, $\beta>0$ fixed.  Then
\begin{equation*}
\lim_{n\to\infty} d_{TV}\bigl(X(\digraph{n}{k}{\alpha}), X(\digraph{n}{k}{\infty})\bigr)=\frac{1}{2}\Ex\lvert1-\exp[-\mathcal{N}]\rvert,
\end{equation*}
where $\mathcal{N}$ is as in Theorem \ref{thm:totalvar-midalpha}.
\end{theorem}

Theorems \ref{thm:totalvar} and \ref{thm:sumofsquares} open an avenue for further study. For instance, suppose $\alpha=\Theta(n^{\sigma})$, $\sigma\in (0,1/2]$. In that case, by Theorem \ref{thm:totalvar-smallalpha}, $1-\totalvar\to 0$. The questions are how fast, and which parameter of $M$  is ``in charge'' of the convergence rate? Suppose $\sigma\in [1/s,1/2]$, $s>2$ being an integer. Introduce 
\begin{equation*}
\mathbf{X}=\mathbf{X}(M)=\{X^{(t)}(M)\}_{t=2}^s, \qquad X^{(t)}(M):=\sum_i (D_i(M))^t.
\end{equation*}
Is it true that $1-d_{TV}\bigl(\mathbf{X}(\digraph{n}{k}{\alpha}), 
\mathbf{X}(\digraph{n}{k}{\infty})\bigr)\sim 1-\totalvar$?

\section{Preliminary Results}\label{sec:prelims}
\begin{definition}
Let $\indegseq{n}{\alpha}=(\indeg{n}{1}{\alpha},\ldots,\indeg{n}{n}{\alpha})$ denote the in-degree sequence for $\digraph{n}{k}{\alpha}$, and let $\indegseq{n}{\infty}$ denote the in-degree sequence for $\digraph{n}{k}{\infty}$.
\end{definition}
Note that $\vec{d}=(d_1,\ldots,d_n)$ is an admissible in-degree sequence precisely when it satisfies $d_1,\ldots,d_n\geq 0$ and $d_1+\cdots+d_n=kn$. (From now on, we use $\vec{d}$ to denote a generic admissible sequence.) The number of mappings with in-degree sequence $\vec{d}$ is precisely $\binom{kn}{\vec{d}}$, which is shorthand for the multinomial coefficient
\begin{equation*}
\binom{kn}{\vec{d}}:=\binom{kn}{d_1,\ldots,d_n}.
\end{equation*}

The coordinates of $\indegseq{n}{\alpha}$ are interdependent, as $\sum_j\indeg{n}{j}{\alpha}=kn$.  However, there are IID random variables that can be gainfully used to analyze these coordinates.  For $k=1$, it was proved in \cite{hansen-jaworski-exchindegrees} that the in-degrees are (jointly) distributed as IID negative binomial variables, conditioned on summing to $n$; likewise, the in-degrees of the uniformly random mapping are distributed as IID Poisson variables, conditioned on summing to $n$.  These results generalize to $\digraph{n}{k}{\alpha}$ and $\digraph{n}{k}{\infty}$:
\begin{lemma}[In-degree sequence distributions]\label{lem:indeg-dists}
Let $\vec{d}=(d_1,d_2,\ldots,d_n)$ be given. 
\begin{sublemma}
\item\label{lem:indeg-dists-alpha} Let $Z_{n,1},\ldots,Z_{n,n}$ be IID random variables with the generalized negative binomial distribution with shape parameter $\alpha$ and probability $\frac{k}{\alpha+k}$:
\begin{equation*}
P(Z_{n,j}=d)=\frac{\alpha^{\rising{d}}}{d!}\left(\frac{\alpha}{\alpha+k}\right)^{\alpha}\left(\frac{k}{\alpha+k}\right)^{d},\qquad d=0,1,2,\ldots.
\end{equation*}
Let $\vec{Z}_n=(Z_{n,1},\ldots,Z_{n,n})$. Then
\begin{equation*}
P(\indegseq{n}{\alpha}=\vec{d})=P(\vec{Z}_n=\vec{d}\mid Z_{n,1}+\cdots+Z_{n,n}=kn).
\end{equation*}
\item\label{lem:indeg-dists-unif}Let $Y_{n,1},\ldots,Y_{n,n}$ be IID Poisson-distributed random variables with mean $k$. Let $\vec{Y}_n=(Y_{n,1},\ldots,Y_{n,n})$. Then
\begin{equation*}
P(\indegseq{n}{\infty}=\vec{d})=P(\vec{Y}_n=\vec{d}\mid Y_{n,1}+\cdots+Y_{n,n}=kn).
\end{equation*}
\end{sublemma}
\end{lemma}
\begin{proof}
The probability generating function for $Z_{n,j}$ is
\begin{equation}\label{eq:pgf-negbinom}
f_{Z}(x):=\Ex[x^{Z_{n,j}}]=\left(\frac{\alpha}{\alpha+k}\right)^{\alpha}\left(1-\frac{kx}{\alpha+k}\right)^{-\alpha}.
\end{equation}
It follows by independence of the $Z_{n,j}$ that
\begin{equation*}
P\left(\sum_j Z_{n,j}=kn\right)=[x^{kn}](f_Z(x))^n=\frac{(\alpha n)^{\rising{kn}}}{(kn)!}\left(\frac{\alpha}{\alpha+k}\right)^{\alpha n}\left(\frac{k}{\alpha+k}\right)^{kn},
\end{equation*}
while independence and $\sum_j d_j=kn$ imply
\begin{equation*}
P(\vec{Z}_n=\vec{d})=\left(\frac{\alpha}{\alpha+k}\right)^{n\alpha}\left(\frac{k}{\alpha+k}\right)^{kn}\prod_{j=1}^{n}\frac{\alpha^{\rising{d_j}}}{d_j!}.
\end{equation*}
Combining these yields
\begin{equation*}
P(\vec{Z}_n=\vec{d}\mid Z_{n,1}+\cdots+Z_{n,n}=kn)=\binom{kn}{\vec{d}}\frac{\prod_{j=1}^{n}\alpha^{\rising{d_j}}}{(\alpha n)^{\rising{kn}}}.
\end{equation*}
The multinomial coefficient is precisely the number of $k$-out mappings on $[n]$, and so \eqref{eq:distribution} implies
\begin{equation*}
P(\indegseq{n}{\alpha}=\vec{d})=\binom{kn}{\vec{d}}\frac{\prod_{j=1}^{n}\alpha^{\rising{d_j}}}{(\alpha n)^{\rising{kn}}}=P\left(\vec{Z}_n=\vec{d}\mid \sum_j Z_{n,j}=kn\right),
\end{equation*}
proving (i). The proof of (ii) proceeds in the same fashion. 
\end{proof}

We will find that the total variation distance we seek can be computed by focusing on the in-degree sequences of our mappings; to do so, we will need information about the moments of the in-degrees, and some results about concentration. With Lemma \ref{lem:indeg-dists} in hand, the moments can be computed explicitly: 
\begin{corollary}[Moments]\label{cor:moments}
Let $n,k\in\mathbb{N}$ and $\alpha\in(0,\infty)$.
\begin{subcorollary}
\item\label{cor:moments-unconditioned} The factorial moments and moments of $Z_{n,j}$ (defined as in Lemma \ref{lem:indeg-dists-alpha}) are
\begin{equation*}
\Ex[(Z_{n,j})_{\ell}]=\frac{\alpha^{\rising{\ell}}k^{\ell}}{\alpha^{\ell}}\qquad\text{and}\qquad\Ex[(Z_{n,j})^s]=\sum_{\ell=1}^{s}\stirling{s}{\ell}\frac{\alpha^{\rising{\ell}}k^{\ell}}{\alpha^{\ell}},
\end{equation*}
where $\stirling{s}{\ell}$ is the Stirling partition number and $(a)_b=a(a-1)\cdots(a-(b-1))$ is the falling factorial.
\item\label{cor:moments-factorial} The factorial moments of $\indeg{n}{j}{\alpha}$ and $\indeg{n}{j}{\infty}$ are, respectively,
\begin{equation*}
\Ex[(\indeg{n}{j}{\alpha})_{\ell}]=\frac{\alpha^{\rising{\ell}}(kn)_{\ell}}{(\alpha n)^{\rising{\ell}}}\qquad\text{and}\qquad\Ex[(\indeg{n}{j}{\infty})_{\ell}]=\frac{(kn)_{\ell}}{n^{\ell}}.
\end{equation*}
\item\label{cor:moments-pure} The moments of $\indeg{n}{j}{\alpha}$ and $\indeg{n}{j}{\infty}$ are, respectively,
\begin{equation*}
\mu_{s,\alpha}:=\Ex[(\indeg{n}{j}{\alpha})^s]=\sum_{\ell=1}^{s}\stirling{s}{\ell}\frac{\alpha^{\rising{\ell}}(kn)_{\ell}}{(\alpha n)^{\rising{\ell}}}
\end{equation*}
and
\begin{equation*}
\mu_{s,\infty}:=\Ex[(\indeg{n}{j}{\infty})^s]=\sum_{\ell=1}^{s}\stirling{s}{\ell}\frac{(kn)_{\ell}}{n^{\ell}}.
\end{equation*}
\item\label{cor:moments-mixedfactorial} The mixed factorial moments for $\indegseq{n}{\alpha}$ and $\indegseq{n}{\infty}$ are, respectively,
\begin{equation*}
\Ex[(\indeg{n}{i}{\alpha})_{\ell}(\indeg{n}{j}{\alpha})_m]=\frac{\alpha^{\rising{\ell}}\alpha^{\rising{m}}(kn)_{\ell+m}}{(\alpha n)^{\rising{\ell+m}}}\qquad\text{and}\qquad\Ex[(\indeg{n}{i}{\infty})_{\ell}(\indeg{n}{j}{\infty})_m]=\frac{(kn)_{\ell+m}}{n^{\ell+m}}.
\end{equation*}
\item\label{cor:moments-alpha-asymptotics} If $\alpha=\alpha(n)\to\infty$ as $n\to\infty$, then
\begin{equation*}
\mu_{s,\alpha}=\mu_{s,\infty}+O\left(\frac{1}{\alpha}\right).
\end{equation*}
\item\label{cor:moments-unif-asymptotics} As $n\to\infty$,
\begin{equation*}
\mu_{s,\infty}=\sum_{\ell=1}^{s}\stirling{s}{\ell}k^{\ell}+O\left(\frac{1}{n}\right).
\end{equation*}
\end{subcorollary}
\end{corollary}
\begin{proof}
Let $\vec{Z}_n=(Z_{n,1},\ldots,Z_{n,n})$ and $\vec{Y}_n=(Y_{n,1},\ldots,Y_{n,n})$ be as in Lemma \ref{lem:indeg-dists}, and let $f_Z(x)$ be the probability generating function for $Z_{n,j}$ as computed in \eqref{eq:pgf-negbinom}. Then $\Ex[(Z_{n,j})_{\ell}]=f_Z^{(\ell)}(1)$, and the first formula in (i) follows. The proof of (i) is completed by the identity
\begin{equation}\label{identity}
x^{s}=\sum_{\ell=1}^{s}\stirling{s}{\ell}(x)_{\ell}.
\end{equation}
Note that because $Z_{n,1},\ldots,Z_{n,n}$ are IID,
\begin{align*}
\Ex[(\indeg{n}{j}{\alpha})_{\ell}]&=\sum_{d=\ell}^{kn}(d)_{\ell}P(\indeg{n}{j}{\alpha}=d)\\
&=\sum_{d=\ell}^{kn}\frac{(d)_{\ell} P(Z_{n,1}=d)\,P\bigl(\sum_{j\ge 2}Z_{n,j}=kn-d\bigr)}{
P\bigl(\sum_{j\ge 1}Z_{n,j}=kn)}\\
&=\frac{[x^{kn}](x^{\ell}f_Z^{(\ell)}(x))(f_Z(x))^{n-1}}{[x^{kn}](f_Z(x))^n}
=\frac{\alpha^{\rising{\ell}}(kn)_{\ell}}{(\alpha n)^{\rising{\ell}}}.
\end{align*}
Arguing similarly for the uniform mapping yields
\begin{equation*}
\Ex[(\indeg{n}{j}{\infty})_{\ell}]=\frac{[x^{kn}](x^{\ell}f_Y^{(\ell)}(x))(f_Y(x))^{n-1}}{[x^{kn}](f_Y(x))^n}=\frac{(kn)_{\ell}}{n^{\ell}},
\end{equation*}
where $f_Y(x)=e^{k(x-1)}$ is the probability generating function for $Y_{n,j}$. This completes the proof of (ii). Claim (iii) follows from (ii) and the identity \eqref{identity}. For (iv): arguing as in the proof of (ii) leads to
\begin{equation*}
\Ex[(\indeg{n}{i}{\alpha})_{\ell}(\indeg{n}{j}{\alpha})_m]=\frac{[x^{kn}](x^{\ell}f_{Z}^{(\ell)}(x))\,(x^{m}f_{Z}^{(m)}(x))\,(f_Z(x))^{n-2}}{[x^{kn}](f_Z(x))^{n}}=\frac{\alpha^{\rising{\ell}}\alpha^{\rising{m}}(kn)_{\ell+m}}{(\alpha n)^{\rising{\ell+m}}},
\end{equation*}
while
\begin{equation*}
\Ex[(\indeg{n}{i}{\infty})_{\ell}(\indeg{n}{j}{\infty})_m]=\frac{[x^{kn}](x^{\ell}f_{Y}^{(\ell)}(x))\,(x^{m}f_{Y}^{(m)}(x))\,(f_Y(x))^{n-2}}{[x^{kn}](f_Y(x))^{n}}=\frac{(kn)_{\ell+m}}{n^{\ell+m}}.
\end{equation*}
For (v), we need only notice that if $\alpha\to\infty$, then
\begin{equation*}
\mu_{s,\alpha}=\sum_{\ell=1}^{s}\stirling{s}{\ell}\frac{\alpha^{\rising{\ell}}(kn)_{\ell}}{(\alpha n)^{\rising{\ell}}}=\sum_{\ell=1}^{s}\stirling{s}{\ell}\frac{\alpha^{\ell}(kn)_{\ell}}{(\alpha n)^{\ell}}\left(1+O\left(\frac{1}{\alpha}\right)\right)=\mu_{s,\infty}+O\left(\frac{1}{\alpha}\right).
\end{equation*}
Finally, (vi) is an immediate consequence of the expression for $\mu_{s,\infty}$ in (iii).  This completes the proof.
\end{proof}

The most important applications of Corollary \ref{cor:moments} are listed in the following  
statement:
\begin{corollary}[Concentration results]\label{cor:concetration}
Suppose $\omega=\omega(n)\to\infty$ as $n\to\infty$, however slowly.  Let $s\in\mathbb{N}$, $\mu_{s,\alpha}:=\Ex[(\indeg{n}{j}{\alpha})^s]$, and $\mu_{s,\infty}:=\Ex[(\indeg{n}{j}{\infty})^s]$.
\begin{subcorollary}
\item\label{cor:concentration-alpha} If $\alpha=\alpha(n)$ is bounded away from $0$, then
\begin{equation*}
\lim_{n\to\infty}P\bigl(\lvert (\indeg{n}{1}{\alpha})^s+\cdots+(\indeg{n}{n}{\alpha})^s-\mu_{s,\alpha}n\rvert<\omega\sqrt{n}\,\bigr)=1.
\end{equation*}
\item\label{cor:concentration-unif} For the uniform map,
\begin{equation*}
\lim_{n\to\infty}P\bigl(\lvert (\indeg{n}{1}{\infty})^s+\cdots+(\indeg{n}{n}{\infty})^s-\mu_{s,\infty}n\rvert<\omega\sqrt{n}\,\bigr)=1.
\end{equation*}
\end{subcorollary}
\end{corollary}
\begin{proof}
Let us first consider (i).  Note that
\begin{equation*}
\Ex[(\indeg{n}{1}{\alpha})^s+\cdots+(\indeg{n}{n}{\alpha})^s]=n\Ex[(\indeg{n}{1}{\alpha})^s]=\mu_{s,\alpha}n.
\end{equation*}
Further, the moments in Corollary \ref{cor:moments} imply that
\begin{align}\label{eq:sumofsquares1}
\Ex\Bigl[\Bigl(\sum_{j=1}^{n}(\indeg{n}{j}{\alpha})^s\Bigr)^2\Bigr]&=n^2\Ex[(\indeg{n}{1}{\alpha})^s(\indeg{n}{2}{\alpha})^s]+O(n).
\end{align}
Rewrite $(\indeg{n}{1}{\alpha})^s(\indeg{n}{2}{\alpha})^s$ in terms of falling factorials, to find
\begin{align*}
\Ex[(\indeg{n}{1}{\alpha})^s(\indeg{n}{2}{\alpha})^s]&=\Ex\Bigl[\Bigl(\sum_{\ell=1}^{s}\stirling{s}{\ell}(\indeg{n}{1}{\alpha})_\ell\Bigr)\Bigl(\sum_{m=1}^{s}\stirling{s}{m}(\indeg{n}{2}{\alpha})_{m}\Bigr)\Bigr]\\
&=\sum_{\ell=1}^{s}\sum_{m=1}^{s}\stirling{s}{\ell}\stirling{s}{m}\frac{\alpha^{\rising{\ell}}\alpha^{\rising{m}}(kn)_{\ell+m}}{(\alpha n)^{\rising{\ell+m}}}\\
&=\sum_{\ell=1}^{s}\sum_{m=1}^{s}\stirling{s}{\ell}\stirling{s}{m}\frac{\alpha^{\rising{\ell}}\alpha^{\rising{m}}k^{\ell+m}}{\alpha^{\ell+m}}\left(1+O\left(\frac{1}{n}\right)\right).
\end{align*}
For $\alpha$ bounded away from $0$, the summands here are bounded, so that
\begin{align*}
\Ex[(\indeg{n}{1}{\alpha})^s&(\indeg{n}{2}{\alpha})^s]=\sum_{\ell=1}^{s}\sum_{m=1}^{s}\stirling{s}{\ell}\stirling{s}{m}\frac{\alpha^{\rising{\ell}}k^{\ell}}{\alpha^{\ell}}\cdot\frac{\alpha^{\rising{m}}k^{m}}{\alpha^m}+O\left(\frac{1}{n}\right)\\
=&\,\left(\sum_{\ell=1}^{s}\stirling{s}{\ell}\frac{\alpha^{\rising{\ell}}k^{\ell}}{\alpha^{\ell}}\right)^2+O\left(\frac{1}{n}\right)\\
=&\left(\sum_{\ell=1}^{s}\stirling{s}{\ell}\frac{\alpha^{\rising{\ell}}(kn)_{\ell}}{(\alpha n)^{\rising{\ell}}}+O\left(\frac{1}{n}\right)\right)^2+O\left(\frac{1}{n}\right)
=(\mu_{s,\alpha})^2+O\left(\frac{1}{n}\right).
\end{align*}
This, combined with \eqref{eq:sumofsquares1}, yields
$\Var\Bigl[\sum_{j=1}^{n}(\indeg{n}{j}{\alpha})^s\Bigr]=O(n)$.
Result (i) follows immediately via Chebyshev's inequality, and (ii) is proved similarly.
\end{proof}

The last ingredients that we need before moving on to prove Theorem \ref{thm:totalvar} are the following bounds on the rising factorial:
\begin{lemma}[Rising factorial bounds]\label{lem:risingfac-bounds}
Suppose $a\in(0,\infty)$ and $b\in\mathbb{Z}\cap[0,a+1)$.  Then:
\begin{sublemma}
\item\label{lem:risingfac-bounds-rough} The rising factorial $a^{\rising{b}}$ satisfies
\begin{equation*}
\exp\left(-\frac{b^3}{6a^2}\right)\leq\frac{a^{\rising{b}}}{a^b\exp\left(\frac{b(b-1)}{2a}\right)}\leq1.
\end{equation*}
\item\label{lem:risingfac-bounds-sharp} The rising factorial $a^{\rising{b}}$ satisfies
\begin{equation*}
1\leq\frac{a^{\rising{b}}}{a^b\exp\left(\frac{b(b-1)}{2a}-\frac{b(b-1)(2b-1)}{12a^2}\right)}\leq\exp\left(\frac{b^4}{12a^3}\right).
\end{equation*}
\end{sublemma}
\end{lemma}
\begin{proof}
Write
\begin{equation}\label{eq:risingfac-bound}
a^{\rising{b}}=a^b\exp\left(\sum_{j=0}^{b-1}\log\left(1+\frac{j}{a}\right)\right).
\end{equation}
For $x\in(0,1)$, the power series for $\log(1+x)$ has alternating terms which decrease in absolute value, so that $\log(1+x)$ is sandwiched between any two successive partial sums.  From this, we get the bounds
\begin{equation}\label{eq:logbound-1}
\frac{j}{a}-\frac{j^2}{2a^2}\leq\log\left(1+\frac{j}{a}\right)\leq\frac{j}{a},
\end{equation}
and
\begin{equation}\label{eq:logbound-2}
\frac{j}{a}-\frac{j^2}{2a^2}\leq\log\left(1+\frac{j}{a}\right)\leq\frac{j}{a}-\frac{j^2}{2a^2}+\frac{j^3}{3a^3}
\end{equation}
for $0\leq j\leq b-1<a$.  It follows from \eqref{eq:risingfac-bound} and \eqref{eq:logbound-1} that
\begin{equation*}
\exp\left(-\frac{b^3}{6a^2}\right)\le \exp\left(-\frac{b(b-1)(2b-1)}{12a^2}\right)\leq\frac{a^{\rising{b}}}{a^b\exp\left(\frac{b(b-1)}{2a}\right)}\leq 1,
\end{equation*}
which proves  Part (i).  Part (ii) follows similarly from \eqref{eq:risingfac-bound} and \eqref{eq:logbound-2}.
\end{proof}

\section{Proof of Theorem \ref{thm:totalvar-bigalpha}: The Case $\alpha\gg\sqrt{n}$}
We are now ready to prove Theorem \ref{thm:totalvar-bigalpha} -- that $\totalvar\to0$ as $n\to\infty$ if $\alpha/\sqrt{n}\to\infty$. The first step is to rewrite the total variation distance in terms of in-degree sequences, using \eqref{eq:distribution}. Letting $\koutmapsdegs{n}{k}{\vec{d}}$ denote the collection of all $k$-out maps on $[n]$ with in-degree sequence $\vec{d}$, we compute
\begin{align}
\totalvar&=\frac{1}{2}\sum_{\vec{d}}\sum_{M\in\koutmapsdegs{n}{k}{\vec{d}}}\lvert P(\digraph{n}{k}{\alpha}=M)-P(\digraph{n}{k}{\infty}=M)\rvert\notag\\
&=\frac{1}{2}\sum_{\vec{d}}\sum_{M\in\koutmapsdegs{n}{k}{\vec{d}}}\left\lvert\frac{\prod_{j=1}^{n}\alpha^{\rising{d_j}}}{(\alpha n)^{\rising{kn}}}-\frac{1}{n^{kn}}\right\rvert\notag\\
&=\frac{1}{2}\sum_{\vec{d}}\binom{kn}{\vec{d}}\left\lvert\frac{\prod_{j=1}^{n}\alpha^{\rising{d_j}}}{(\alpha n)^{\rising{kn}}}-\frac{1}{n^{kn}}\right\rvert\label{eq:totalvar-rewrite},
\end{align}
where the summation is over all valid in-degree sequences $\vec{d}=(d_1,\ldots,d_n)$: $d_i\geq 0$ for all $i$, and $d_1+\cdots+d_n=kn$. Notice that the right side of \eqref{eq:totalvar-rewrite} is precisely the total variation distance between $\indegseq{n}{\alpha}$ and $\indegseq{n}{\infty}$.  We might have expected this:  indeed, conditioned on the in-degree sequence, $\digraph{n}{k}{\alpha}$ and $\digraph{n}{k}{\infty}$ are both uniformly random.

Let us write $\alpha=\omega\sqrt{n}$; note that $\omega=\omega(n)\to\infty$ as $n\to\infty$. Let $\mathcal{B}_n$ denote the collection of in-degree sequences $\vec{d}=(d_1,\ldots,d_n)$ such that
\begin{equation*}
\lvert d_1^2+\cdots+d_n^2-\mu_{2,\alpha} n\rvert<\sqrt{\omega n}\qquad\text{and}\qquad\lvert d_1^3+\cdots+d_n^3-\mu_{3,\alpha}n\rvert<\sqrt{\omega n},
\end{equation*}
where as before $\mu_{s,\alpha}:=\Ex[(\indeg{n}{j}{\alpha})^s]$. We split the sum from \eqref{eq:totalvar-rewrite} into major and minor contributions:
\begin{equation*}
\frac{1}{2}\sum_{\vec{d}}\binom{kn}{\vec{d}}\left\lvert\frac{\prod_{j=1}^{n}\alpha^{\rising{d_j}}}{(\alpha n)^{\rising{kn}}}-\frac{1}{n^{kn}}\right\rvert=\frac{1}{2}\sum_{\vec{d}\in\mathcal{B}_n}\binom{kn}{\vec{d}}\left\lvert\frac{\prod_{j=1}^{n}\alpha^{\rising{d_j}}}{(\alpha n)^{\rising{kn}}}-\frac{1}{n^{kn}}\right\rvert+\Sigma_n.
\end{equation*}
For $\Sigma_n$,  by the triangle inequality,
\begin{align}
\Sigma_n&=\frac{1}{2}\sum_{\vec{d}\notin\mathcal{B}_n}\binom{kn}{\vec{d}}\left\lvert\frac{\prod_{j=1}^{n}\alpha^{\rising{d_j}}}{(\alpha n)^{\rising{kn}}}-\frac{1}{n^{kn}}\right\rvert\notag\\
&\leq\sum_{\vec{d}\notin\mathcal{B}_n}\binom{kn}{\vec{d}}\frac{\prod_{j=1}^{n}\alpha^{\rising{d_j}}}{(\alpha n)^{\rising{kn}}}+\sum_{\vec{d}\notin\mathcal{B}_n}\binom{kn}{\vec{d}}\frac{1}{n^{kn}}\notag\\
&=P(\indegseq{n}{\alpha}\notin\mathcal{B}_n)+P(\indegseq{n}{\infty}\notin\mathcal{B}_n).\label{eq:bigalpha-minor}
\end{align}
By Corollary \ref{cor:concentration-alpha}, $P(\indegseq{n}{\alpha}\notin\mathcal{B}_n)\to0$ as $n\to\infty$. By Corollary \ref{cor:moments-alpha-asymptotics}, 
\begin{equation*}
\lvert\mu_{s,\alpha}-\mu_{s,\infty}\rvert=O\left(\frac{1}{\alpha}\right)=O\left(\frac{1}{\omega\sqrt{n}}\right),\qquad s\in\{2,3\},
\end{equation*}
so that
\begin{equation*}
\lvert n\mu_{s,\alpha}-n\mu_{s,\infty}\rvert=O\left(\frac{\sqrt{n}}{\omega}\right)=o(\sqrt{\omega n}).
\end{equation*}
It follows that $P(\indegseq{n}{\infty}\notin\mathcal{B}_n)\to0$, since by Corollary \ref{cor:concentration-unif}, with its $\omega$ replaced by $\sqrt{\omega}$, we have
\begin{equation*}
P\Bigl(\Bigl\lvert\sum_{j=1}^{n}(\indeg{n}{j}{\infty})^s-\mu_{s,\infty}n\Bigr\rvert<\sqrt{\omega n}\,\text{ for } s=2,3\Bigr)\to1\text{ as }n\to\infty.
\end{equation*}
So, by \eqref{eq:bigalpha-minor}, $\Sigma_n\to0$ as $n\to\infty$, and we can focus on the sum over $\vec{d}\in\mathcal{B}_n$.

Applying the rising factorial estimate in Lemma \ref{lem:risingfac-bounds-rough}, we find that
\begin{align}\label{eq:bigalpha-totalweight}
(\alpha n)^{\rising{kn}}&=(\alpha n)^{kn}\exp\left(\frac{kn(kn-1)}{2\alpha n}\right)\left(1+O\left(\frac{(kn)^3}{(\alpha n)^2}\right)\right)\notag\\
&=(\alpha n)^{kn}\exp\left(\frac{k^2\sqrt{n}}{2\omega}\right)\left(1+O\left(\frac{1}{\omega^2}\right)+O\left(\frac{1}{\omega\sqrt{n}}\right)\right).
\end{align}
Using the same bounds for each factor $\alpha^{\rising{d_j}}$ shows that, {\em uniformly over all} $\vec{d}$,
\begin{equation}\label{eq:bigalpha-weightasymptotics}
\prod_{j=1}^{n}\alpha^{\rising{d_j}}=\alpha^{kn}\exp\biggl(\sum_{j=1}^{n}\frac{d_j(d_j-1)}{2\alpha}+O\biggl(\sum_{j=1}^{n}\frac{d_j^3}{\alpha^2}\biggr)\biggr).
\end{equation}
Here, uniformly over $\vec{d}\in\mathcal{B}_n$, 
\begin{equation*}
\sum_{j=1}^{n}\frac{d_j(d_j-1)}{2\alpha}=\frac{\mu_{2,\alpha} n-kn+O\left(\sqrt{\omega n}\right)}{2\alpha}.
\end{equation*}
Further, Corollaries \ref{cor:moments-alpha-asymptotics} and \ref{cor:moments-unif-asymptotics} imply
\begin{equation*}
\mu_{2,\alpha}n=\mu_{2,\infty}n+O\left(\frac{n}{\alpha}\right)=(k^2+k)n+O(1)+O\left(\frac{n}{\alpha}\right),
\end{equation*}
so that
\begin{equation*}
\sum_{j=1}^{n}\frac{d_j(d_j-1)}{2\alpha}=\frac{k^2n}{2\alpha}+O\left(\frac{1}{\omega\sqrt{n}}\right)+O\left(\frac{1}{\omega^2}\right)+O\left(\frac{1}{\sqrt{\omega}}\right)=\frac{k^2\sqrt{n}}{2\omega}+O\left(\frac{1}{\sqrt{\omega}}\right).
\end{equation*}
Such $\vec{d}$ also satisfy
\begin{equation*}
\sum_{j=1}^{n}\frac{d_j^3}{\alpha^2}=\frac{\mu_{3,\alpha}n+O(\sqrt{\omega n})}{\alpha^2}=O\left(\frac{1}{\omega^2}\right).
\end{equation*}
So, returning to \eqref{eq:bigalpha-weightasymptotics}, we find that uniformly over $\vec{d}\in\mathcal{B}_n$,
\begin{equation}\label{eq:bigalpha-weightasymptotics2}
\prod_{j=1}^{n}\alpha^{\rising{d_j}}=\alpha^{kn}\exp\biggl(\frac{k^2\sqrt{n}}{2\omega}\biggr)\left(1+O\left(\frac{1}{\sqrt{\omega}}\right)\right).
\end{equation}
Therefore
\begin{align*}
\frac{1}{2}\sum_{\vec{d}\in\mathcal{B}_n}\binom{kn}{\vec{d}}\left\lvert\frac{\prod_{j=1}^{n}\alpha^{\rising{d_j}}}{(\alpha n)^{\rising{kn}}}-\frac{1}{n^{kn}}\right\rvert&=O\left(\frac{1}{\sqrt{\omega}\,n^{kn}}\sum_{\vec{d}}\binom{kn}{\vec{d}}\right)=O\left(\frac{1}{\sqrt{\omega}}\right)\to 0,
\end{align*}
as $n\to\infty$. This completes the proof.

\section{Proof of Theorem \ref{thm:totalvar-smallalpha}: The Case $\alpha\ll\sqrt{n}$}\label{sec:smallalpha}
Having seen that $\totalvar\to0$ as $n\to\infty$ when $\alpha/\sqrt{n}\to\infty$, the natural question to ask is this: is this the best we can do? We now prove that it is, in the sense that $\totalvar\to1$ as $n\to\infty$, if $\alpha\to\infty$ but $\alpha/\sqrt{n}\to0$.

Note that for any event $\mathcal{A}_n\subseteq\koutmaps{n}{k}$,
\begin{align*}
\totalvar&=\sup_{\mathcal{A}\subseteq\koutmaps{n}{k}}\lvert P(\digraph{n}{k}{\alpha}\in\mathcal{A})-P(\digraph{n}{k}{\infty}\in\mathcal{A})\rvert\\
&\geq\lvert P(\digraph{n}{k}{\alpha}\in\mathcal{A}_n)-P(\digraph{n}{k}{\infty}\in\mathcal{A}_n)\rvert.
\end{align*}
As such, it is enough to find an event $\mathcal{A}_n$ such that  $P(\digraph{n}{k}{\alpha}
\in \mathcal{A}_n)\to 1$ and  $P(\digraph{n}{k}{\infty}\in\mathcal{A}_n)\to 0$. To that end, let us write $\alpha=\sqrt{n}/\omega$, and let $\mathcal{G}_n$ denote the set of all valid in-degree sequences $\vec{d}=(d_1,\ldots,d_n)$ such that
\begin{equation*}
\lvert d_1^2+\cdots+d_n^2-\mu_{2,\alpha}n\rvert<\sqrt{\omega n}.
\end{equation*} 
Then the event $\{M: \mathbf{d}(M)\in \mathcal{G}_n\}$  is precisely the event
$\mathcal{A}_n$ that we seek. Indeed, since $\omega\to\infty$ as $n\to\infty$, $P(\indegseq{n}{\alpha}\in\mathcal{G}_n)\to1$ as $n\to\infty$ by Corollary \ref{cor:concentration-alpha}. On the other hand,  Corollary \ref{cor:concentration-unif} states that
\begin{equation}\label{eq:bigalpha-unif-concentration}
P(\lvert (\indeg{n}{1}{\infty})^2+\cdots+(\indeg{n}{n}{\infty})^2-\mu_{2,\infty}n\rvert
<\sqrt{\omega n})\to1\text{ as }n\to\infty.
\end{equation}
Here, by Corollary \ref{cor:moments-pure},
\begin{equation*}
\lvert\mu_{2,\alpha}n-\mu_{2,\infty}n\rvert=\frac{k(n-1)(kn-1)}{\alpha n+1}\sim\frac{k^2n}{\alpha}=k^2\omega\sqrt{n},
\end{equation*}
whence $\lvert\mu_{2,\alpha}n-\mu_{2,\infty}n\rvert\geq2\sqrt{\omega n}$ for $n$ sufficiently large, so that \eqref{eq:bigalpha-unif-concentration} implies that $P(\indegseq{n}{\infty}\in\mathcal{G}_n)\to0$ as $n\to\infty$. We conclude that
\begin{equation*}
1=\lim_{n\to\infty}\lvert P(\indegseq{n}{\alpha}\in\mathcal{G}_n)-P(\indegseq{n}{\infty}\in\mathcal{G}_n)\rvert\leq\totalvar\leq 1,
\end{equation*}
thus completing the proof.

\section{Proof of Theorem \ref{thm:totalvar-midalpha}: The Case $\alpha=\beta\sqrt{n}$}\label{sec:midalpha}
We have now seen that $\totalvar\to 0$ if $\alpha/\sqrt{n}\to\infty$, meaning that $\digraph{n}{k}{\alpha}$ is, in a strong sense, asymptotically uniform in this case.  We have also seen that $\totalvar\to1$ if $\alpha/\sqrt{n}\to0$, so that, in
the limit,  the supports of $\digraph{n}{k}{\alpha}$ and $\digraph{n}{k}{\infty}$ in
$\mathcal{M}_{n,k}$ are disjoint. The natural follow-up question, of course, is this: what happens in between? We are now ready to prove Theorem \ref{thm:totalvar-midalpha}, which covers exactly this case. Let us start by giving a basic outline of the proof.

A slight manipulation of \eqref{eq:totalvar-rewrite} gives us that
\begin{equation}\label{eq:totalvar-rewrite2}
\totalvar=\frac{1}{2}\sum_{\vec{d}}\binom{kn}{d}\frac{\prod_{j=1}^{n}\alpha^{\rising{d_j}}}{(\alpha n)^{\rising{kn}}}\left\lvert1-\frac{(\alpha n)^{\rising{kn}}}{n^{kn}\prod_{j=1}^{n}\alpha^{\rising{d_j}}}\right\rvert.
\end{equation}
Note that this expresses the total variation distance as the expectation, with respect to $M_{n,k}^{\alpha}$, of the following quantity:
\begin{equation}\label{eq:midalpha-heuristics-summand}
\frac{1}{2}\left\lvert1-\frac{(\alpha n)^{\rising{kn}}}{n^{kn}\prod_{j=1}^{n}\alpha^{\rising{\indeg{n}{j}{\alpha}}}}\right\rvert.
\end{equation}
As in the proof of Theorem \ref{thm:totalvar-bigalpha}, we begin by splitting the sum in \eqref{eq:totalvar-rewrite2} into major and minor contributions of ``good'' and ``bad'' $\mathbf{d}$, largely corresponding to two complementary ranges of $\sum_j d_j^2$. We will show that the contribution of bad $\mathbf{d}$ is negligible; so our focus will be on good $\mathbf{d}$. A sharp asymptotic analysis will show that, uniformly  over  those $\vec{d}$,
\begin{equation}\label{eq:midalpha-heuristics-approx}
\left\lvert1-\frac{(\alpha n)^{\rising{kn}}}{n^{kn}\prod_{j=1}^{n}\alpha^{\rising{d_j}}}\right\rvert\approx f(\mathcal{S}_n^0(\mathbf{d})),\quad \mathcal{S}_n^0(\mathbf{d}):=
\frac{\sum_{j=1}^{n}d_j^2-n\Ex[(Z_{n,1})^2]}{\sqrt{n}};
\end{equation}
here $\vec{Z}_n=(Z_{n,1},\ldots,Z_{n,n})$ is as in Lemma \ref{lem:indeg-dists-alpha} and
$
f(x):=\left\lvert1-\exp\left(-\frac{k^2}{4\beta^2}-\frac{x}{2\beta}\right)\right\rvert.
$
Since $\Ex[(Z_{n,1})^2]$ and $\Ex[(\indeg{n}{1}{\alpha})^2]$ are relatively close, \eqref{eq:midalpha-heuristics-approx} is a clear sign that a central limit theorem (CLT) for $(\indeg{n}{1}{\alpha})^2+\cdots+(\indeg{n}{n}{\alpha})^2$ might be the key for asymptotic analysis of the contribution by good $\mathbf{d}$. In this setting, a CLT is indeed plausible: we know that $\indegseq{n}{\alpha}$ is distributed as $\vec{Z}_n$ conditioned on $Z_{n,1}+\cdots+Z_{n,n}=kn$, a condition weak enough that any {\it fixed, bounded} set of coordinates of $\indegseq{n}{\alpha}$ are asymptotically independent.

Still, the interdependence of $D_{n,j}^{\alpha}$ is too strong to count on standard techniques, such as Fourier- and/or martingale-based approaches. If workable at all, the method of moments would have required sharp asymptotic estimates of the {\it central} moments of $(\indeg{n}{1}{\alpha})^2+\cdots+(\indeg{n}{n}{\alpha})^2$, obtained from an extension of Corollary \ref{cor:moments-mixedfactorial} to higher-order mixed factorial moments -- an exceedingly computational route which would hardly explain the intrinsic reasons for the CLT to hold.

So, instead, we recall the result of Lemma \ref{lem:indeg-dists-alpha}: the in-degree sequence $\indegseq{n}{\alpha}$ is distributed as $\vec{Z}_n=(Z_{n,1},\ldots,Z_{n,n})$ conditioned on $Z_{n,1}+\cdots+Z_{n,n}=kn$, where the $Z_{n,j}$ are IID negative binomial variables. In light of this, we consider a two-dimensional  $\vec{S}_n=\sum_j\bigl(\frac{1}{\sqrt{n}}Z_{n,j},\,\frac{1}{\sqrt{n}}Z_{n,j}^2\bigr)$, the sum of independent $2$-vectors, or, more specifically, the centered  vector
\begin{align*}
\vec{S}_n^{0}=(S_{n,1}^0,S_{n2}^0)&=\left(\frac{\sum_{j=1}^{n}Z_{n,j}-n\Ex[Z_{n,1}]}{\sqrt{n}},\,\frac{\sum_{j=1}^{n}Z_{n,j}^2-n\Ex[Z_{n,1}^2]}{\sqrt{n}}\right).
\end{align*}
Conditioned on $S_{n1}^0=0$, $S_{n2}^0$ is distributed as $\mathcal{S}_n^0(\indegseq{n}{\alpha})$, $\mathcal{S}_n^0(\mathbf{d})$ being defined in \eqref{eq:midalpha-heuristics-approx}. Now we should certainly expect that $\vec{S}_n^{0}$ is asymptotically Gaussian (normal).  However just a CLT for $\vec{S}_n^{0}$ would not be enough, since $P(S_{n1}^0=0)=\Theta(n^{-1/2})$, making the conditioning event $\{S_{n1}^0=0\}$ way too intrusive to extract the limiting distribution of  $\mathcal{S}_n^0(\indegseq{n}{\alpha})$ from  the limiting {\it cumulative} distribution of $\vec{S}_n^{0}$. So, instead, we will have to prove a Local Central Limit Theorem (LCLT) for $\vec{S}_n^{0}$; this will immediately yield a LCLT (which in turn directly implies a CLT) for $\mathcal{S}_n^0(\indegseq{n}{\alpha})$.
\begin{lemma}[LCLT for $\vec{S}_n^0$]\label{lem:unconditioned-llt}
Suppose $\alpha=\alpha(n)<\infty$ for all $n$, but that $\alpha(n)\to\infty$ as $n\to\infty$.  
Then
\begin{equation*}
\lim_{n\to\infty}\sup_{\vec{x}\in\Supp(\vec{S}_n^0)}\left\lvert\frac{n}{2}P\left(\vec{S}_n^0=\vec{x}\right)-\eta(\vec{x})\right\rvert=0,
\end{equation*}
where
\begin{equation*}
\eta(\vec{x}):=\frac{\exp(-\frac{1}{2}\vec{x}\mathbf{\Sigma}^{-1}\vec{x}^T)}{2\pi\sqrt{\lvert\det\mathbf{\Sigma}\rvert}}
\end{equation*}
is the density function of a Gaussian random vector in $\mathbb{R}^2$ with mean $\vec{0}$ and the positive-definite covariance matrix $\mathbf{\Sigma}$ and its inverse $\mathbf{\Sigma}^{-1}$
given by
\begin{equation*}
\mathbf{\Sigma}=\begin{bmatrix}k & 2k^2+k\\ 2k^2+k & 4k^3+6k^2+k\end{bmatrix},\quad
\mathbf{\Sigma}^{-1}=\begin{bmatrix}2+3k^{-1}+k^{-2}/2 & -k^{-1}-k^{-2}/2\\-k^{-1}-k^{-2}/2 & k^{-2}/2
\end{bmatrix}.
\end{equation*}
\end{lemma}
\begin{proof}
As a template, we use the Fourier-based proof of the one-variable LCLT in  Durrett \cite[Section 2.5, Theorem 5.2]{durrett}, modifying it for vector-valued summands, and dealing with dependence on $n$ (through dependence of $\alpha$ on $n$) of the distributions of the individual summands. Most of the relevant facts about multi-dimensional characteristic functions, such as inversion formulas, can be found in \cite{bhattacharya-rao-1976} by Bhattacharya and Rao.

For ease of notation, define
\begin{equation*}
\vec{V}_{n,j}:=(Z_{n,j}-\Ex[Z_{n,j}],Z_{n,j}^2-\Ex[Z_{n,j}^2]),\quad \vec{V}_n=\sum_j \vec{V}_{n,j};
\end{equation*}
so $\vec{V}_{n,j}$ are IID, and $\vec{S}_n^0=n^{-1/2}\vec{V}_n$. For $\vec{t}\in \mathbb{R}^2$, introduce the characteristic functions
\begin{equation*}
\phi_n(\vec{t}):=\Ex[e^{i\langle\vec{t},\vec{V}_{n,j}\rangle}],\qquad \Phi_n(\vec{t})=\Ex[e^{i\langle\vec{t},\vec{S}_n^0\rangle}]=(\phi_n(n^{-1/2}\vec{t}))^n.
\end{equation*}
Let us show first that 
\begin{equation}\label{Phintto}
\lim_{n\to\infty} \Phi_n(\vec{t})=e^{-\frac{1}{2}\vec{t}\vec{\Sigma}\vec{t}^{\transpose}},\quad\vec{t}
\in \mathbb{R}^2;
\end{equation}
see Lemma \ref{lem:unconditioned-llt} for $\vec{\Sigma}$. This will prove the CLT for $\vec{S}_n^0$, namely: $\vec{S}_n^0$ converges in distribution to a Gaussian $2$-vector with mean $\vec{0}$ and covariance matrix $\vec{\Sigma}$.

Each  $\vec{V}_{n,j}$ has mean $\vec{0}$, and the same covariance matrix, $\vec{\Sigma}_n$.  By Corollary \ref{cor:moments-pure}, $\Ex\bigl[\|\vec{V}_{n,j}\|^3\bigr] =O(1)$ as $n\to\infty$. So, uniformly in $n$,
\begin{equation}\label{eq:midalpha-charasymp}
\phi_n(\vec{y})=1-\vec{y}\vec{\Sigma}_n\vec{y}^{\transpose}/2+O\left(\lvert\vec{y}\rvert^3\right), \quad \vec{y}\to\vec{0}.
\end{equation}
Using Corollary \ref{cor:moments-unconditioned} to compute $\vec{\Sigma}_n$, we find that  $\|\vec{\Sigma}_n-\vec{\Sigma}\|=O(\alpha^{-1})\to 0$,  as $\alpha=\Theta(n^{1/2})$. Thus, for every $\vec t$,
\begin{equation*}
\Phi_n(\vec{t})=\bigl(\phi_n(n^{-1/2}\vec{t})\bigr)^n=\bigl(1-n^{-1}\vec{t}\vec{\Sigma}\vec{t}^{\transpose}/2
+O(n^{-3/2})\bigr)^n\to e^{-\frac{1}{2}\vec{t}\vec{\Sigma}\vec{t}^{\transpose}},
\end{equation*}
which proves \eqref{Phintto}.

The CLT is a key ingredient in the proof of the LCLT for $\vec{S}_n^0$. 

The minimum additive subgroup $\mathcal{L}\subseteq\mathbb{R}^2$ such that $P(\vec{V}_{n,j}\in\vec{x}_0+\mathcal{L})=1$ for some $\vec{x}_0\in\mathbb{R}^2$ is generated by $\vec{a}=(1,1)$ and $\vec{b}=(1,-1)$: $\Supp\,(Z_{n,j},Z_{n,j}^2)$ must be contained in the span of $\vec{a}$ and $\vec{b}$,
because $m^2\equiv m\pmod{2}$ for all non-negative integers $m$, and $\vec{a}$, $\vec{b}$ are necessary because $(0,0),(1,1),(2,4)\in\Supp(Z_{n,j},Z_{n,j}^2)$. Since the random $\vec{V}_{n,j}$ are independent, the minimum subgroup for $\vec{S}_n^0$ is $\mathcal{L}/\sqrt{n}$, the lattice generated by $\vec{a}/\sqrt{n}$ and $\vec{b}/\sqrt{n}$. So, by the inversion formula for lattice-distributed variables (c.f. \cite[p.230, eq. 21.28]{bhattacharya-rao-1976}),
\begin{equation}\label{eq:inversion-discrete}
P(\vec{S}_n^0=\vec{x})=\frac{1}{4\pi^2}\cdot\frac{2}{n}\cdot\int_{\mathcal{F}_n}e^{-i\langle\vec{t},\vec{x}\rangle}\Phi_n(\vec{t})\,d\vec{t}, \qquad\vec{x}\in\Supp\,(\vec{S}_n^0),
\end{equation}
where
\begin{equation*}
\mathcal{F}_n:=\{\vec{t}=(t_1,t_2)\st \lvert t_1+t_2\rvert<\pi\sqrt{n}\text{ and }\lvert t_1-t_2\rvert<\pi\sqrt{n}\}.
\end{equation*}
Since the characteristic function  $e^{-\frac{1}{2}\vec{t}\vec{\Sigma}\vec{t}^{\transpose}}\in 
L_1(\mathbb{R}^2)$, we also have 
\begin{equation}\label{eq:inversion-continuous}
\eta(\vec{x})=
\frac{1}{4\pi^2}\int_{\mathbb{R}^2}e^{-i\langle\vec{t},\vec{x}\rangle}e^{-\frac{1}{2}\vec{t}\vec{\Sigma}\vec{t}^{\transpose}}\,d\vec{t}.
\end{equation}
The triangle inequality and equations \eqref{eq:inversion-discrete} and \eqref{eq:inversion-continuous} imply
\begin{equation}\label{eq:midalpha-triangleineq}
\left\lvert\frac{n}{2}P(\vec{S}_n^0=\vec{x})-\eta(\vec{x})\right\rvert\leq\int_{\mathcal{F}_n}\left\lvert\Phi_n(\vec{t})-e^{-\frac{1}{t}\vec{t}\vec{\Sigma}\vec{t}^{\transpose}}\right\rvert\,d\vec{t}+\int_{\mathbb{R}^2\setminus\mathcal{F}_n}e^{-\frac{1}{2}\vec{t}\vec{\Sigma}\vec{t}^{\transpose}}\,d\vec{t}.
\end{equation}
The right side of \eqref{eq:midalpha-triangleineq} does not depend on $\vec{x}$.  So, in order to establish uniform convergence, it suffices to show that the right side of \eqref{eq:midalpha-triangleineq} tends to $0$ as $n\to\infty$.

That the integral over $\mathbb{R}^2\setminus\mathcal{F}_n$ tends to $0$ is immediate: $e^{-\frac{1}{2}\vec{t}\vec{\Sigma}\vec{t}^{\transpose}}$ is integrable, and $\mathcal{F}_n$ increases to $\mathbb{R}^2$ as $n\to\infty$. Consider the integral over $\mathbb{R}^2\setminus\mathcal{F}_n$.
Note that, by the CLT already proved, the integrand converges to zero pointwise. 
To prove that the integral goes to zero as well, we consider $\vec{t}$ with
small $n^{-1/2}|\vec{t}|$ and the remaining $\vec{t}$ separately. 

We know that  $\vec{\Sigma}_n\to \vec{\Sigma}$,  and $\vec{\Sigma}$ is positive-definite. 
So there is a constant $\gamma>0$ such that, for $n$ large enough,
$\vec{y}\vec{\Sigma}_n\vec{y}^{\transpose}\geq\gamma\lvert\vec{y}\rvert^2$ for all $\vec{y}\in\mathbb{R}^2$.  Consequently, by the uniform estimate in \eqref{eq:midalpha-charasymp}, there is a constant $\delta^\prime>0$ such that for $n$ large enough and $|\vec{y}|\le\delta^\prime$,
\begin{equation*}
\lvert\phi_n(\vec{y})\rvert\leq1-\frac{1}{4}\gamma\lvert\vec{y}\rvert^2\leq e^{-\frac{1}{4}\gamma\lvert\vec{y}\rvert^2}.
\end{equation*}
Choose $\delta\in(0,\pi)$ so that $\vec{t}\in\mathcal{F}_n^{(1)}$ implies $\lvert\vec{t}\rvert<\delta'\sqrt{n}$, where
\begin{equation*}
\mathcal{F}_n^{(1)}:=\{\vec{t}\st \lvert t_1+t_2\rvert<\delta\sqrt{n}\text{ and }\lvert t_1-t_2\rvert<\delta\sqrt{n}\}.
\end{equation*}
Then for $\vec{t}\in\mathcal{F}_n^{(1)}$,
\begin{equation*}
\lvert\Phi_n(\vec{t})-e^{-\frac{1}{2}\vec{t}\vec{\Sigma}\vec{t}^{\transpose}}\rvert\leq\lvert\Phi_n(\vec{t})\rvert+e^{-\frac{1}{2}\vec{t}\vec{\Sigma}\vec{t}^{\transpose}}\leq(e^{-\frac{1}{4}\gamma\lvert\vec{t}\rvert^2/n})^n+e^{-\frac{1}{2}\vec{t}\vec{\Sigma}\vec{t}^{\transpose}}=e^{-\frac{1}{4}\gamma\lvert\vec{t}\rvert^2}+e^{-\frac{1}{2}\vec{t}\vec{\Sigma}\vec{t}^{\transpose}},
\end{equation*}
which is integrable. So, by the Dominated Convergence Theorem,
\begin{equation*}
\int_{\mathcal{F}_n^{(1)}}\lvert\Phi_n(\vec{t})-e^{-\frac{1}{2}\vec{t}\vec{\Sigma}\vec{t}^{\transpose}}\rvert\,d\vec{t}\to0\text{ as }n\to\infty.
\end{equation*}
For $\mathcal{F}_n^{(2)}:=\mathcal{F}_n\setminus\mathcal{F}_n^{(1)}$, note that $\vec{V}_{n,j}$ converges in distribution to $(Y-k, Y^2-(k^2+k))$ as $n\to\infty$, where $Y$ is Poisson-distributed with parameter $k$. Therefore, $\phi_n(\vec{t})$ converges to the characteristic function $\phi^*(\vec{t})$ of $(Y,Y^2)$, uniformly on compact sets in $\mathbb{R}^2$. By \cite[Lemma 21.6]{bhattacharya-rao-1976}, $\lvert\phi^*(\vec{y})\rvert=1$ if and only if $\vec{y}=(y_1,y_2)$ has $\lvert y_1+y_2\rvert,\lvert y_1-y_2\rvert\in2\pi\mathbb{Z}$. We note that, uniformly for
$\vec{t}\in\mathcal{F}_n^{(2)}$,  $\vec{t}/\sqrt{n}$ is bounded away from all such points.  So, by uniform continuity of $\phi^*$ on every compact set, there exists $\epsilon\in(0,1)$ so that for all $n$ and for all $\vec{t}\in\mathcal{F}_n^{(2)}$, $\lvert\phi^*(\vec{t}/\sqrt{n})\rvert\le\epsilon/2$. Using
again the fact that $\mathcal{F}_n^{(2)}/\sqrt{n}$ is contained in a compact set, we know that $\phi_n(\vec{t})$ converges uniformly to $\phi^*(\vec{t})$ on $\mathcal{F}_n^{(2)}/\sqrt{n}$.  Thus $\lvert\phi_n(\vec{t}/\sqrt{n})\rvert\le\epsilon$ and $\lvert\Phi_n(\vec{t})\rvert\le\epsilon^n$ for $\vec{t}\in\mathcal{F}_n^{(2)}$ and $n$ sufficiently large. It then follows that for $n$ sufficiently large,
\begin{equation*}
\int_{\mathcal{F}_n^{(2)}}\lvert\Phi_n(\vec{t})-e^{-\frac{1}{2}\vec{t}\vec{\Sigma}\vec{t}^{\transpose}}\rvert\,d\vec{t}\leq\epsilon^n\text{Vol}\,(\mathcal{F}_n^{(2)})+\int_{\mathcal{F}_n^{(2)}}e^{-\frac{1}{2}\vec{t}\vec{\Sigma}\vec{t}^{\transpose}}\,d\vec{t}\to0,
\end{equation*}
 as $n\to\infty$, completing the proof.
\end{proof}

With Lemma \ref{lem:unconditioned-llt} in hand, we are ready to prove the desired CLT for the sum of squared in-degrees.  In fact, the LCLT for $\vec{S}_n^{0}$ is strong enough to prove the corresponding LCLT for that sum, which directly implies  the desired convergence in distribution.
\begin{corollary}[LCLT for $\sum_j(\indeg{n}{j}{\alpha})^2$]\label{cor:llt}
Suppose $\alpha=\alpha(n)<\infty$ for all $n$, but $\alpha(n)\to\infty$ as $n\to\infty$. Define
\begin{equation*}
S_n^0=\frac{\sum_j(\indeg{n}{j}{\alpha})^2 - n\Ex\bigl[(Z_{n,1}^{\alpha})^2\bigr]}{\sqrt{n}}.
\end{equation*}
Then 
\begin{equation*}
\lim_{n\to\infty}\sup_{x\in\Supp(S_n^0)}\left\lvert\frac{\sqrt{n}}{2}P(S_n^0=x)-\psi(x)\right\rvert=0,
\end{equation*}
where $\psi(x):=\frac{e^{-x^2/(4k^2)}}{2k\sqrt{\pi}}$
is the density of a Gaussian random variable with mean $0$ and variance $2k^2$.
\end{corollary}
\begin{proof}
Let $\vec{S}_n^0$ be defined as in Lemma \ref{lem:unconditioned-llt}. Then for any 
$x\in\Supp\,(S_n^0)$,
\begin{multline}\label{PSn0=x}
P(S_n^0=x)=P\biggr(\vec{S}_n^0=(0,x)\mid \sum_j Z_{n,j}=kn\biggr)\\
=\frac{P(\vec{S}_n^0=(0,x),\ \sum_j Z_{n,j}=kn)}{P(\sum_j Z_{n,j}=kn)}
=\frac{P(\vec{S}_n^0=(0,x))}{P(\sum_j Z_{n,j}=kn)},
\end{multline}
since the condition $\sum_j Z_{n,j}=kn$ means precisely that the first coordinate of $\vec{S}_n^0$ 
is zero. As in the proof of Lemma \ref{lem:indeg-dists}, 
\begin{equation*}
P\left(\sum_j Z_{n,j}=kn\right)=\frac{(\alpha n)^{\rising{kn}}}{(kn)!}\left(\frac{\alpha}{\alpha+k}\right)^{\alpha n}\left(\frac{k}{\alpha+k}\right)^{kn}.
\end{equation*}
We apply the identity $a^{\rising{b}}=\Gamma(a+b)/\Gamma(a)$ and Stirling's approximation to obtain
\begin{align*}
P\left(\sum_j Z_{n,j}=kn\right)&=\sqrt{\frac{\alpha}{2\pi k(\alpha+k)n}}\left(1+O\left(\frac{1}{n}\right)\right)\\
&=\frac{1}{\sqrt{2\pi k n}}\left(1+O\left(\frac{1}{\alpha}\right)+O\left(\frac{1}{n}\right)\right).
\end{align*}
So, using \eqref{PSn0=x} and Lemma \ref{lem:unconditioned-llt}, uniformly over 
$x\in\Supp(S_n^0)$,
\begin{align*}
\frac{\sqrt{n}}{2}P(S_n^0=x)=&\,\bigl(1+O(\alpha^{-1})\bigr)\sqrt{2\pi k}\cdot\frac{n}{2}P(\vec{S}_n^0=(0,x))\\
=&\,\bigl(1+O(\alpha^{-1})\bigr)\sqrt{2\pi k}\,(\eta(0,x)+o(1))=\psi(x) +o(1),
\end{align*}
where
\begin{equation*}
\psi(x)=\sqrt{2\pi k}\,\eta(0,x)=\frac{\sqrt{2\pi k}}{2\pi\sqrt{\lvert\det\mathbf{\Sigma}\rvert}}\,
\exp\left(-\frac{x^2}{2}(\mathbf{\Sigma}^{-1})_{1,1}\right)
=\frac{e^{-x^2/(2k^2)}}{2k\sqrt{\pi}}.
\end{equation*}
So
$$
\lim_{n\to\infty}\sup_{x\in\Supp(S_n^0)}\left\lvert\frac{\sqrt{n}}{2}P(S_n^0=x)-\psi(x)\right\rvert=0,
$$
which completes the proof.
\end{proof} 
Finally, we are ready to carry out the proof of Theorem \ref{thm:totalvar-midalpha}. We proceed via a series of claims, after making some initial definitions.

Define
\begin{equation*}
f(x):=\left\lvert1-\exp\left(-\frac{k^2}{4\beta^2}-\frac{x}{2\beta}\right)\right\rvert.
\end{equation*}
Let $\omega=\omega(n)$ be such that $\omega\to\infty$ as $n\to\infty$ but $\omega/\sqrt{n}\to0$. For $A>0$, define the function $f_A$ by
\begin{equation*}
f_A(x):=\begin{cases}f(x) & \text{if }x\geq-A,\\f(-A) & \text{otherwise};\end{cases}
\end{equation*}
so $f_A(x)$ is bounded and continuous for any fixed $A$. Further, define
\begin{equation*}
\mathcal{A}_n=\mathcal{A}_n(A):=\Bigl\{\vec{d}\st\Bigl\lvert\sum_{j=1}^{n}d_j^2-n\Ex[(\indeg{n}{j}{\alpha})^2]\Bigr\rvert<A\sqrt{n}\Bigr\},
\end{equation*}
where $\vec{d}$ ranges over all valid in-degree sequences of $k$-out digraphs on $[n]$, and
\begin{equation*}
\mathcal{A}_n'=\mathcal{A}_n'(A):=\Bigl\{\vec{d}\in\mathcal{A}_n\st\Bigl\lvert\sum_{j=1}^{n}d_j^s-n\Ex[(\indeg{n}{j}{\alpha})^s]\Bigr\rvert<\omega\sqrt{n}\text{ for }s=3,4\Bigr\}.
\end{equation*}
As in the proof of Theorem \ref{thm:totalvar-bigalpha}, we start from
\begin{equation*}
\totalvar=\frac{1}{2}\sum_{\vec{d}}\binom{kn}{\vec{d}}\left\lvert\frac{\prod_{j=1}^{n}\alpha^{\rising{d_j}}}{(\alpha n)^{\rising{kn}}}-\frac{1}{n^{kn}}\right\rvert.
\end{equation*}

\begin{proposition}\label{prop:1}
The error committed in restricting $\totalvar$ to in-degree \linebreak sequences $\vec{d}\in\mathcal{A}_n'$ can be made small by choosing $A$ large. In particular,  for $A>\frac{2k^2}{\beta}$, there is a constant $C>0$, independent of $A$, such that 
\begin{equation*}
\limsup_{n\to\infty}\sum_{\vec{d}\notin\mathcal{A}_n'}\binom{kn}{\vec{d}}\left\lvert\frac{\prod_{j=1}^{n}\alpha^{\rising{d_j}}}{(\alpha n)^{\rising{kn}}}-\frac{1}{n^{kn}}\right\rvert\leq\frac{C}{(A-\frac{k^2}{\beta})^2}.
\end{equation*}
\end{proposition}
\begin{proof}
By the triangle inequality,
\begin{equation*}
\sum_{\vec{d}\notin\mathcal{A}_n'}\binom{kn}{\vec{d}}\left\lvert\frac{\prod_{j=1}^{n}\alpha^{\rising{d_j}}}{(\alpha n)^{\rising{kn}}}-\frac{1}{n^{kn}}\right\rvert\leq P(\indegseq{n}{\alpha}\notin\mathcal{A}_n')+P(\indegseq{n}{\infty}\notin\mathcal{A}_n').
\end{equation*}
By Corollary \ref{cor:concentration-alpha},
\begin{equation*}
P\Bigl(\Bigl\lvert\sum_{j=1}^{n}(\indeg{n}{j}{\alpha})^s-n\Ex[(\indeg{n}{1}{\alpha})^s]\Bigr\rvert\geq \omega\sqrt{n}\Bigr)\to0\text{ as }n\to\infty
\end{equation*}
for $s=3$ and $s=4$, while a similar application of Chebyshev's inequality yields
\begin{equation*}
P\Bigl(\Bigl\lvert\sum_{j=1}^{n}(\indeg{n}{j}{\alpha})^2-n\Ex[(\indeg{n}{1}{\alpha})^2]\Bigr\rvert\geq A\sqrt{n}\Bigr)\leq\frac{\Var[\sum_{j=1}^{n}(\indeg{n}{j}{\alpha})^2]}{A^2n}=O(A^{-2}).
\end{equation*}
The last two estimates combine to imply existence of a constant $C_1>0$, independent
of  $A$, so that 
\begin{equation}\label{C1}
\limsup_{n\to\infty}P(\indegseq{n}{\alpha}\notin\mathcal{A}_n')\leq\frac{C_1}{A^2}.
\end{equation}
For $n$ sufficiently large, Corollary \ref{cor:moments-factorial} and $\alpha=\beta n^{1/2}$ imply that, for $s=2,3,4$, 
$$
\lvert\Ex[(\indeg{n}{j}{\alpha})^s]-\Ex[(\indeg{n}{j}{\infty})^s]\rvert\le c_s n^{-1/2},
$$
$c_s$ being constants, with $c_2=2k^2/\beta$. 
By Chebyshev's inequality, for $s=3$ and $s=4$ we have
\begin{equation*}
P\Bigl(\Bigl\lvert\sum_{j=1}^{n}(\indeg{n}{j}{\infty})^s-n\Ex[(\indeg{n}{j}{\alpha})^s]\Bigr\rvert\geq\omega\sqrt{n}\Bigr)\leq\frac{\Var\bigl[\,\sum_{j=1}^{n}(\indeg{n}{j}{\infty})^s\bigr]}{(\omega-
c_s)^2n}\to0
\end{equation*}
as $n\to\infty$, while for some constant $C_2$ independent of $A$
\begin{equation*}
P\Bigl(\Big\lvert\sum_{j=1}^{n}(\indeg{n}{j}{\infty})^2-n\Ex[(\indeg{n}{j}{\alpha})^2]\Bigr\rvert\geq A\sqrt{n}\Bigr)\leq\frac{\Var\bigl[\,\sum_{j=1}^{n}(\indeg{n}{j}{\infty})^2\bigr]}{(A-\frac{2k^2}{\beta})^2n}\leq\frac{C_2}{(A-\frac{2k^2}{\beta})^2}.
\end{equation*}
This bound  and  \eqref{C1}  combined imply  that
\begin{equation*}
\limsup_{n\to\infty}P(\indegseq{n}{\infty}\notin\mathcal{A}_n')\leq\frac{C_1+C_2}{(A-\frac{2k^2}{\beta})^2},
\end{equation*}
completing the proof.
\end{proof}

\begin{proposition}\label{prop:2}
As $n\to\infty$,
\begin{equation*}
\frac{1}{2}\sum_{\vec{d}\in\mathcal{A}_n'}\binom{kn}{\vec{d}}\left\lvert\frac{\prod_{j=1}^{n}\alpha^{\rising{d_j}}}{(\alpha n)^{\rising{kn}}}-\frac{1}{n^{kn}}\right\rvert=\frac{1}{2}\Ex[f(S_n^0)\mathbf{1}_{\indegseq{n}{\alpha}\in\mathcal{A}_n'}]+O\left(\frac{\omega}{\sqrt{n}}\right).
\end{equation*}
\end{proposition}
\begin{proof}
As in the proof of Theorem \ref{thm:totalvar-bigalpha}, we rewrite
\begin{equation*}
\binom{kn}{\vec{d}}\left\lvert\frac{\prod_{j=1}^{n}\alpha^{\rising{d_j}}}{(\alpha n)^{\rising{kn}}}-\frac{1}{n^{kn}}\right\rvert=\binom{kn}{\vec{d}}\frac{\prod_{j=1}^{n}\alpha^{\rising{d_j}}}{(\alpha n)^{\rising{kn}}}\left\lvert1-\frac{(\alpha n)^{\rising{kn}}}{n^{kn}\prod_{j=1}^{n}\alpha^{\rising{d_j}}}\right\rvert.
\end{equation*}
By  Lemma \ref{lem:risingfac-bounds-sharp},
\begin{equation*}
(\alpha n)^{\rising{kn}}=(\alpha n)^{kn}\exp\left(\frac{k^2n}{2\beta\sqrt{n}}-\frac{k^3}{6\beta^2}+O\left(\frac{1}{\sqrt{n}}\right)\right).
\end{equation*}
 Applying Lemma \ref{lem:risingfac-bounds-sharp}to each $\alpha^{\rising{d_j}}$ and using $\mu_{s,\alpha}$ from Corollary \ref{cor:moments-pure}
we obtain: uniformly over all $\vec{d}\in\mathcal{A}_n'$,
\begin{align*}
\prod_{j=1}^{n}\alpha^{\rising{d_j}}&=\alpha^{kn}\exp\left[\sum_{j=1}^{n}\left(\frac{d_j(d_j-1)}{2\alpha}-\frac{d_j(d_j-1)(2d_j-1)}{12\alpha^2}\right)+O\left(\sum_{j=1}^{n}\frac{d_j^4}{\alpha^3}\right)\right]\\
&=\alpha^{kn}\exp\left[\sum_{j=1}^{n}\left(\frac{d_j^2}{2\alpha}-\frac{d_j}{2\alpha}-\frac{2d_j^3-3d_j^2+d_j}{12\alpha^2}\right)+O\left(\frac{1}{\sqrt{n}}\right)\right]\\
&=\alpha^{kn}\exp\left[\frac{\sum_{j=1}^{n}d_j^2}{2\beta\sqrt{n}}-\frac{k\sqrt{n}}{2\beta}-\frac{2\mu_{3,\alpha}-3\mu_{2,\alpha}+k}{12\beta^2}+O\left(\frac{\omega}{\sqrt{n}}\right)\right]\\
&=\alpha^{kn}\exp\left[\frac{\sum_{j=1}^{n}d_j^2-n\Ex[Z_{n,1}^2]}{2\beta\sqrt{n}}+\frac{k^2n}{2\beta\sqrt{n}}-\frac{k^3}{6\beta^2}+\frac{k^2}{4\beta^2}+O\left(\frac{\omega}{\sqrt{n}}\right)\right].
\end{align*}
Combining these results yields
\begin{equation}\label{eq:midalpha-alphaasymp}
\frac{(\alpha n)^{\rising{kn}}}{n^{kn}\prod_{j=1}^{n}\alpha^{\rising{d_j}}}=\exp\left[-\frac{k^2}{4\beta^2}-\frac{\sum_{j=1}^{n}d_j^2-n\Ex[Z_{n,1}^2]}{2\beta\sqrt{n}}+O\left(\frac{\omega}{\sqrt{n}}\right)\right].
\end{equation}
Thus, denoting $x(\vec{d})=n^{-1/2}\bigl(\sum_j d_j^2-n\Ex[Z_{n,1}^2]\bigr)$ and recalling the definition
of $f$,
\begin{align*}
\left\lvert1-\frac{(\alpha n)^{\rising{kn}}}{n^{kn}\prod_{j=1}^{n}\alpha^{\rising{d_j}}}\right\rvert&=
f(x(\vec{d}))+O\bigl(e^{-k^2/4\beta^2-x(\vec{d})}\omega n^{-1/2}\bigr)\\
&=f(x(\vec{d}))+O(\omega n^{-
1/2}),
\end{align*}
uniformly over $\vec{d}\in\mathcal{A}_n'$, because $|x(\vec{d})|$ is bounded over such $\vec{d}$. It follows that
\begin{equation*}
\sum_{\vec{d}\in\mathcal{A}_n'}\binom{kn}{\vec{d}}\frac{\prod_{j=1}^{n}\alpha^{\rising{d_j}}}{(\alpha n)^{\rising{kn}}}\left\lvert1-\frac{(\alpha n)^{\rising{kn}}}{n^{kn}\prod_{j=1}^{n}\alpha^{\rising{d_j}}}\right\rvert=\Ex[f(S_n^0)\mathbf{1}_{\{\indegseq{n}{\alpha}\in\mathcal{A}_n'\}}]+O\left(\frac{\omega}{\sqrt{n}}\right),
\end{equation*}
as claimed.
\end{proof}
\begin{proposition}\label{prop:3}
There is a constant $C$ (independent of $A$) such that
\begin{equation*}
\limsup_{n\to\infty}\left\lvert\frac{1}{2}\Ex[f(S_n^0)\mathbf{1}_{\indegseq{n}{\alpha}\in\mathcal{A}_n'}]-\frac{1}{2}\Ex[f_A(S_n^0)]\right\rvert\leq C\frac{(1+e^{A/2\beta})e^{-A^2/4k^2}}{A}.
\end{equation*}
\end{proposition}
\begin{proof}
Consider the expectations as sums over $\vec{d}$ and match the summands for the
same $\vec{d}$. The matched summands are identical for $\vec{d}\in\mathcal{A}_n'$, and
the summand from the first sum is $0$ for $\vec{d}\notin\mathcal{A}_n'$.  So
we only need to bound $\Ex[f_A(S_n^0)\mathbf{1}_{\{\indegseq{n}{\alpha}\notin\mathcal{A}_n'\}}]$,
where  $\{\indegseq{n}{\alpha}\notin\mathcal{A}_n'\}\subseteq 
\{\indegseq{n}{\alpha}\in\mathcal{A}_n\setminus\mathcal{A}_n'\} \cup
\{\indegseq{n}{\alpha}\notin\mathcal{A}_n\}$. Since $\lvert f_A(x)\rvert\leq 1+e^{A/2\beta}$,
\begin{equation}\label{An-An'}
\Ex[f_A(S_n^0)\mathbf{1}_{\{\indegseq{n}{\alpha}\in\mathcal{A}_n\setminus\mathcal{A}_n'\}}]\leq(1+e^{A/2\beta})P(\indegseq{n}{\alpha}\notin\mathcal{A}_n')
=O\left(\frac{\omega}{\sqrt{n}}\right),
\end{equation}
and, since $S_n^0\Rightarrow\mathcal{N}(0,2k^2)$ as $n\to\infty$ by Corollary \ref{cor:llt}
and $A$ is fixed,
\begin{align}
\Ex[ f_A(S_n^0)\mathbf{1}_{\{\indegseq{n}{\alpha}\notin\mathcal{A}_n\}}]&\leq(1+e^{A/2\beta})P(\lvert S_n^0\rvert\geq A)\notag\\
&\to(1+e^{A/2\beta})P(\lvert\mathcal{N}(0,2k^2)\rvert\geq A).\label{notAn}
\end{align}
Here, by  the  tail inequality for normal variables (c.f. \cite[Theorem 1.1.4]{durrett}),
\begin{equation}\label{tail}
P(\lvert\mathcal{N}(0,2k^2)\rvert\geq A)\leq\frac{2k}{\sqrt{\pi}}\cdot\frac{e^{-A^2/4k^2}}{A}.
\end{equation}
Combining  \eqref{An-An'}, \eqref{notAn} and \eqref{tail} proves the claim.
\end{proof}
\begin{proposition}\label{prop:4}
As $n\to\infty$,
\begin{equation*}
\frac{1}{2}\Ex[f_A(S_n^0)]\to\frac{1}{2}\Ex[f_A(\mathcal{N}(0,2k^2))].
\end{equation*}
\begin{proof}
This is an immediate consequence of two facts: that $f_A$ is bounded and continuous, and the CLT for $S_n^0$.
\end{proof}
\end{proposition}

We are now ready to put the pieces together:
\begin{proof}[Proof of Theorem \ref{thm:totalvar-midalpha}]
Combining Propositions 1-4, we find that for $A>\frac{2k^2}{\beta}$,
\begin{multline}\label{eq:midalpha-last}
\limsup_{n\to\infty}\left\lvert\totalvar-\frac{1}{2}\Ex[f_A(\mathcal{N}(0,2k^2))]\right\rvert\\\leq_b\frac{1}{(A-\frac{2k^2}{\beta})^2}+\frac{(1+e^{A/2\beta})e^{-A^2/4k^2}}{A},
\end{multline}
where ``$\leq_b$'' means that the left side is bounded by a constant (independent of $A$) multiple of the right side. Now, letting $A\to\infty$, the right side of \eqref{eq:midalpha-last} tends to $0$. Further, the nonnegative $f_A(x)$ increases pointwise to $f(x)$ for all $x$, so that 
\begin{equation*}
\frac{1}{2}\Ex[f_A(\mathcal{N}(0,2k^2))]\to\frac{1}{2}\Ex[f(\mathcal{N}(0,2k^2)]=\frac{1}{2}\Ex\left\lvert1-\exp\left(-\mathcal{N}\left(\frac{k^2}{4\beta^2},\frac{k^2}{2\beta^2}\right)\right)\right\rvert
\end{equation*}
as $A\to\infty$, proving the result.
\end{proof}

\section{Proof of Theorem \ref{thm:sumofsquares}}
We begin by noting that the total variation distance we seek is
\begin{align*}
\totalvarsquares&=\sum_{s=k^2n}^{k^2n^2}\,\Biggl\lvert\,\sum_{d_1^2+\cdots+d_n^2=s}\binom{kn}{\vec{d}}\left(\frac{\prod_{j=1}^n\alpha^{\rising{d_j}}}{(\alpha n)^{\rising{kn}}}-\frac{1}{n^{kn}}\right)\Biggr\rvert\\
&=\sum_{s=k^2n}^{k^2n^2}\left\lvert\sum_{d_1^2+\cdots+d_n^2=s}\binom{kn}{\vec{d}}\frac{\prod\alpha^{\rising{d_j}}}{(\alpha n)^{\rising{kn}}}\left(1-\frac{(\alpha n)^{\rising{kn}}}{n^{kn}\prod\alpha^{\rising{d_j}}}\right)\right\rvert.
\end{align*}
This differs from the distance between $\digraph{n}{k}{\alpha}$ and $\digraph{n}{k}{\infty}$ only in the placement of the absolute values: were the absolute values inside the inner sum, the two
distances would be exactly the same. So, our task is to show that the triangle inequality is {\em asymptotically sharp}. This requires that we find analogs for Propositions \ref{prop:1} and \ref{prop:2}; however, the rest of the proof of Theorem \ref{thm:totalvar-midalpha} will transfer over directly from there.

Let $\mathcal{A}_n=\mathcal{A}_n(A)$ and $\mathcal{A}_n'=\mathcal{A}_n'(A)$ be as in Section \ref{sec:midalpha}; let $A_{n,s}'$ be the set of all $\vec{d}\in\mathcal{A}_n'$ with $d_1^2+\cdots+d_n^2=s$. If $\mathcal{B}_{n,s}$ denotes the set of all valid $\vec{d}$ with $d_1^2+\cdots+d_n^2=s$ but $\vec{d}\notin\mathcal{A}_n'$, then the minor contribution to the total variation distance is 
\begin{equation*}
\sum_{s}\left\lvert\sum_{\vec{d}\in\mathcal{B}_{n,s}}\binom{kn}{\vec{d}}\frac{\prod\alpha^{\rising{d_j}}}{(\alpha n)^{\rising{kn}}}\left(1-\frac{(\alpha n)^{\rising{kn}}}{n^{kn}\prod\alpha^{\rising{d_j}}}\right)\right\rvert\leq\sum_{\vec{d}\notin\mathcal{A}_n'}\binom{kn}{\vec{d}}\frac{\prod\alpha^{\rising{d_j}}}{(\alpha n)^{\rising{kn}}}\left\lvert1-\frac{(\alpha n)^{\rising{kn}}}{n^{kn}\prod\alpha^{\rising{d_j}}}\right\rvert;
\end{equation*}
thus, by Proposition \ref{prop:1},
\begin{equation*}
\limsup_{n\rightarrow\infty}\sum_s\left\lvert\sum_{\vec{d}\in\mathcal{B}_{n,s}}\binom{kn}{\vec{d}}\frac{\prod\alpha^{\rising{d_j}}}{(\alpha n)^{\rising{kn}}}\left(1-\frac{(\alpha n)^{\rising{kn}}}{n^{kn}\prod\alpha^{\rising{d_j}}}\right)\right\rvert\leq\frac{C}{(A-\frac{k^2}{\beta})^2}.
\end{equation*}
 
To handle the major contribution, note that as in the proof of Proposition \ref{prop:2}, and \eqref{eq:midalpha-alphaasymp} in particular,  
\begin{equation}\label{eq:sumofsquares2}
1-\frac{(\alpha n)^{\rising{kn}}}{n^{kn}\prod_{j=1}^{n}\alpha^{\rising{d_j}}}=1-\exp\left(-\frac{k^2}{4\beta^2}-\frac{x(\vec{d})}{2\beta}\right)+O(\omega n^{-1/2})
\end{equation}
uniformly over $\vec{d}\in\mathcal{A}_n'$, where $x(\vec{d})=n^{-1/2}(\sum_jd_j^2-n\Ex[Z_{n,1}^2])$. Notably, $x(\vec{d})$ -- and therefore the major contribution in \eqref{eq:sumofsquares2} -- is determined entirely by $d_1^2+\cdots+d_n^2$. This allows us to write the major contribution as
\begin{equation}\label{eq:sumofsquares3}
\sum_{s}\left\lvert\sum_{\vec{d}\in\mathcal{A}_{n,s}'}\binom{kn}{\vec{d}}\frac{\prod\alpha^{\rising{d_j}}}{(\alpha n)^{\rising{kn}}}\left(1-\exp\left(-\frac{k^2}{4\beta^2}-\frac{s-n\Ex[Z_{n,1}^2]}{2\beta\sqrt{n}}\right)\right)\right\rvert+O(\omega n^{-1/2}),
\end{equation}
where the outer summation is over $\lvert s-n\Ex[Z_{n,1}^2]\rvert<A\sqrt{n}$. Note that the signs of the terms in the inner summation have no dependence on $\vec{d}$ other than through dependence on $s$.  So, for a given $s$, either all terms are positive or all terms are negative.  So, there is no cancellation in this new form of the sum, and the triangle inequality is actually equality.  This allows us to rewrite \eqref{eq:sumofsquares3} as
\begin{multline*}
\sum_{s}\sum_{\vec{d}\in\mathcal{A}_{n,s}'}\binom{kn}{\vec{d}}\frac{\prod\alpha^{\rising{d_j}}}{(\alpha n)^{\rising{kn}}}\left\lvert1-\exp\left(-\frac{k^2}{4\beta^2}-\frac{s-n\Ex[Z_{n,1}^2]}{2\beta\sqrt{n}}\right)\right\rvert+O(\omega n^{-1/2})\\
=\sum_{\vec{d}\in\mathcal{A}_n'}\binom{kn}{\vec{d}}\frac{\prod\alpha^{\rising{d_j}}}{(\alpha n)^{\rising{kn}}}\left\lvert1-\exp\left(-\frac{k^2}{4\beta^2}-\frac{x(\vec{d})}{2\beta}\right)\right\rvert+O(\omega n^{-1/2})\\
=\Ex[f(S_n^0)\mathbf{1}_{\indegseq{n}{\alpha}\in\mathcal{A}_n'}]+O\left(\frac{\omega}{\sqrt{n}}\right).
\end{multline*}
This is precisely the expression in the conclusion of Proposition \ref{prop:2}. From here, the rest of the proof of Theorem \ref{thm:totalvar-midalpha} applies directly.

\section{Afterword}
At the start of Section \ref{sec:model}, we introduced two one-arc-at-a-time processes that terminate in the random $k$-out mapping $\digraph{n}{k}{\alpha}$ after $kn$ steps; in one, we place a fixed order on the out-arcs to be chosen, and choose their images in this order; in the other, at each step we choose a vertex uniformly at random from the currently {\em unsaturated} vertices, and choose its image.  In this paper, our focus was on the total variation distance between the terminal snapshots $\digraph{n}{k}{\alpha}$ and $\digraph{n}{k}{\infty}$; a natural follow-up question is: how does the total variation distance between the two {\em processes}, for $\alpha=\alpha(n)<\infty$ and $\alpha=\infty$, depend on $\alpha(n)$?  What is a threshold behavior for uniformity?

In the fixed-order case, it is unsurprising that the total variation distance between the processes exactly matches that between the terminal snapshots: if we know the throwing order and the terminal snapshot, then we know the entire process.  We might suspect that the threshold for the randomly-ordered process is actually higher than $\alpha=\Theta(\sqrt{n})$, as this process contains more information than the terminal snapshot; however, this is not the case, as can be seen as a consequence of the fact that the terminal snapshot is independent of the throwing order.

Analysis of the total variation distance from a uniform distribution, between two terminal snapshots and/or two processes, could prove an interesting challenge for other preferential attachment models, such as the process $\{G_{\alpha}(n,M)\}$ studied by Pittel \cite{pittel-evolvingbydegrees}. Since the distribution of $G_{\alpha}(n,M)$ is not accessible directly, a good first step might be the relaxed (multigraph) process $MG_{\alpha}(n,M)$. Our suspicion is that there are different thresholds for uniformity of the terminal snapshot and the entire process, with the threshold for uniformity of the entire process being the larger.

{\bf Acknowledgement.} We are grateful to Huseyin Acan, Dan Poole, and Chris Ross for many stimulating and probing discussions of this study during brain-storming meetings of the combinatorial probability working group at The Ohio State University. We would also like to thank the anonymous referee of this work for their helpful critical comments.

\bibliographystyle{apt}
\bibliography{research}

\end{document}